\documentclass[11pt]{amsart}

\usepackage[all]{xy}

\usepackage{amsmath, amssymb, amsthm, latexsym, amscd}
\numberwithin{equation}{section}

\newtheorem{theorem}[subsection]{Theorem}
\newtheorem{prop}[subsection]{Proposition}
\newtheorem{lemma}[subsection]{Lemma}
\newtheorem{corollary}[subsection]{Corollary}

\theoremstyle{definition}
\newtheorem{define}[subsection]{Definition}

\newtheorem{remark}[subsection]{Remark}

\newcommand{\mf}[1]{\mathfrak{{#1}}}
\newcommand{\N}{\mathbb {N}}
\newcommand{\Q}{\mathbb {Q}}
\newcommand{\C}{\mathbb{C}}
\newcommand{\Z}{\mathbb Z}
\newcommand{\af}{\hat {\mathfrak g}}
\newcommand{\Nd}{N^ \circledast}
\newcommand{\A}{\mathcal A}
\newcommand{\V}{\mathbb V}
\newcommand{\B}{\mathbb B}
\newcommand{\ep}{\epsilon}
\newcommand{\ot} {\otimes}
\newcommand{\cd}{\cdot}
\newcommand{\Th}{\Theta}
\newcommand{\h}{\mathcal H \text{om}}
\newcommand{\M}{\mathbb M}

\title{On the semi-regular module and vertex operator algebras}
\author{Minxian Zhu}
\address{Department of Mathematics, Yale University, New Haven, CT 06520}
\email{minxian.zhu@yale.edu}

\begin{document}
\maketitle

\section{Introduction}

The aim of this paper is to give a proof of a conjecture stated in a previous paper by the author ([Z1]). 

Let $\mf g$ be a simple complex Lie algebra, 
$\af$ be the affine Lie algebra and $h^\vee$ be the dual Coxeter number of $\mf g$. 
Let $\mathcal A_{\mf g, k}$ be the vertex algebroid 
associated to $\mf g$ and a complex number $k$,  %(see [GMS1]), 
according to [GMS1], 
we can construct a vertex algebra 
$U\A_{\mf g, k}$,  %from $\A_{\mf g, k}$, 
called the enveloping algebra of $A_{\mf g, k}$. 
Set $\V = U\A_{\mf g, k}$. 
%let $\af$ be the affine Lie algebra of $\mf g$
%and let $h^\vee$ be the dual Coxeter number of $\mf g$. 
It is shown in [AG] and [GMS2] that 
not only $\V$ is a $\af$-representation of level $k$, 
it is also a $\af$-representation of the dual level $\bar k = - 2h^\vee - k$. 
Moreover the two copies of $\af$-actions commute with each other, 
i.e. $\V$ is a $\af_k \oplus \af_{\bar k}$-representation. 

When $k \notin \Q$, 
the vertex operator algebra $\V$ decomposes into 
$$\oplus_{\lambda \in P^+} V_{\lambda, k} \ot V_{\lambda^*, \bar k}$$
as a $\af_k \oplus \af_{\bar k}$-module (see [FS], [Z1]). 
Here $P^+$ is the set of dominant integral weights of $\mf g$, 
$V_{\lambda, k}$ is the Weyl module 
induced from $V_\lambda$, 
the irreducible representation of $\mf g$ with highest weight $\lambda$,  
in level $k$, 
and $V_{\lambda^*, \bar k}$ is induced from $V_\lambda^*$ in the dual level $\bar k$. 
In fact the vertex operators can be constructed 
using intertwining operators and Knizhnik-Zamolodchikov equations (see [Z1]). 

In the case where $k \in \Q$, 
the $\af_k \oplus \af_{\bar k}$-module structure of $\V$ is much more complicated. 
In the present paper, 
we prove a result about the existence of canonical filtrations of $\V$
conjectured at the end of [Z1].  
More precisely we will prove the following. 

\newcounter{foo}
\newtheorem{localtheorem}[foo]{Theorem}

\begin{localtheorem} \label{maintheorem}
\it{
Let $k \in \Q$, $k > - h^\vee$. 
The vertex operator algebra $\V$ 
admits an increasing (resp. a decreasing) filtration 
of $\af_k \oplus \af_{\bar k}$-submodules
with factors isomorphic to 
$$V_{\lambda, k} \ot V_{\lambda, \bar k}^c \quad 
(\text{resp. } V_{\lambda, k}^c \ot V_{\lambda, \bar k}), 
\quad \lambda \in P^+,$$
where $V_{\lambda, \bar k}^c$ is the contragredient module of $V_{\lambda, \bar k}$
defined by the anti-involution: 
$x(n) \mapsto - x(-n)$, $\underline c \mapsto \underline c$ 
of $\af$. 
}
\end{localtheorem}

We need two ingredients to prove the theorem: 
one is the semi-regular module; 
the other is the regular representation of the corresponding quantum group at a root of unity.  

The standard semi-regular module was first introduced by A. Voronov in [V] 
to treat the semi-infinite cohomology of infinite dimensional Lie algebras 
as a two-sided derived functor of a functor that is neither left nor right exact. 
It was also studied rigorously by S. M. Arkhipov. 
He defined the associative algebra semi-infinite cohomology 
in the derived categories' setting (see [A1]), 
and discovered a deep semi-infinite duality 
which generalizes the classical bar duality of graded associative algebras (see [A2]). 

The semi-regular module $S_\gamma$ 
associated to a semi-infinite structure $\gamma$ of $\af$ (see [V]) 
is the semi-infinite analogue of the universal enveloping algebra $U$ of $\af$. 
In particular $S_\gamma$ is a $U$-bimodule, 
and the tensor product $S_\gamma \ot_U \V$ 
becomes a $\af_{- \bar k} \oplus \af_{\bar k}$-representation. 
We will show in Section 3 that $S_\gamma \ot_U \V$ 
can be embedded into $U^*$ as a bisubmodule. 
In fact it is spanned by the matrix coefficients of modules 
from the category $\mathcal O_{\bar k + h^\vee}$,  
defined and studied by Kazhdan and Lusztig in [KL1-4]
for $\bar k < -  h^\vee$. 
%(note that $\bar k + h^\vee = -k - h^\vee < 0 $). 

In the series of papers [KL1-4], 
Kazhdan and Lusztig defined a structure of braided category on $\mathcal O_{\bar k + h^\vee}$, 
and constructed an equivalence between the tensor category $\mathcal O_{\bar k + h^\vee}$
and the category of finite dimensional integrable representations of the quantum group
with parameter $e^{i \pi/ (\bar k+ h^\vee)}$ (in the simply-laced case). 
It motivated the author to study the structure of regular representations  
of the quantum group at roots of unity (see [Z2]). 

One of the main results in [Z2] is that 
the quantum function algebra admits an increasing filtration of (bi)submodules
such that the subquotients are isomorphic to the tensor products of the dual of Weyl modules 
$W_{- \omega_0 \lambda}^* \ot W_\lambda^*$
($\omega_0$ being the longest element in the Weyl group). 
Translating this to the affine Lie algebra, 
it means that $S_\gamma \ot_U \V$ 
admits an increasing filtration of $\af_{ -\bar k} \oplus \af_{\bar k}$-submodules
with factors isomorphic to $V_{ - \omega_0 \lambda, \bar k}^* \ot V_{\lambda, \bar k}^c$. 
Applying the functor $\h_U(S_\gamma, -)$ (see [S, Theorem 2.1]) 
to this filtration of $S_\gamma \ot_U \V$, 
we obtain an increasing filtration of $\af_k \oplus \af_{\bar k}$-submodules 
of the vertex operator algebra $\V$ 
with factors described in Theorem \ref{maintheorem}.
The corresponding decreasing filtration is obtained 
by using the non-degenerate bilinear form on $\V$ constructed in [Z1]. 

The paper is organized as follows: 
In Section 2, we follow [S] to recall the definition of semi-regular module $S_\gamma$ 
and the two functors defined with it.  
In Section 3, we embed $S_\gamma \ot_U \V$ into the dual of $U$ as a (bi)submodule. 
In Section 4, we prove the main theorem about the filtrations of the vertex operator algebra $\V$ 
using results of [Z2]. 

%Section 2 starts here.........................................................................................................
\section{Semi-regular module $S_\gamma$ and equivalence of categories}

The semi-regular module of a graded Lie algebra with a semi-infinite structure 
was first introduced by A. Voronov in [V], 
where it was called the ``standard semijective module''. 
It replaces the universal enveloping algebra (and its dual) in the semi-infinite theory, 
and like the universal enveloping algebra, 
it possesses left and right (semi)regular representations. 
Voronov used semijective complexes and resolutions
to define the semi-infinite cohomology of infinite dimensional Lie algebras 
as a two-sided derived functor of a functor that is intermediate
between the functors of invariants and coinvariants. 

In [A2], S. M. Arkhipov generalized the classical bar duality of graded associative algebras
to give an alternative construction of the semi-infinite cohomology of associative algebras. 
Given a graded associative algebra $A$ with a triangular decomposition, 
he introduced the endomorphism algebra $A^\sharp$ 
of a semi-regular $A$-module $S_A$ (see [A1]). 
In the case where $A$ is the universal enveloping algebra 
of a graded Lie algebra,  
the algebra $A^\sharp$ is also a universal enveloping algebra
of a Lie algebra which differs from the previous one 
by a $1$-dimensional central extension
(determined by the critical $2$-cocycle). 
%Arkhipov proved the equivalence of suitably chosen derived categories 
%of $A$-modules and the modules of a canonical DG-algebra associated to $A$ 
%and its nonpositively graded subalgebra $B$. 
%In particular 
In the affine Lie algebra case, 
he proved that the category of all $\af$-modules with a Weyl filtration in level $k$
is contravariantly equivalent to the analogous category in the dual level $\bar k$. 
This equivalence was obtained directly in [S], 
where W. Soergel used it to find characters of tilting modules of affine Lie algebras and quantum groups. 

Let us recall the definition of the semi-regular module from [S, Theorem 1.3]. 

Let $\mf g$ be a simple complex Lie algebra. 
Let $\af = \mf{g} \ot \C[t,t^{-1}] \oplus \C \underline c$ 
be the affine Lie algebra, 
where the commutator relations are given by
$$[ x (m), y(n) ] = [x, y] (m+n) + m\delta_{m+n, 0} (x, y) \underline c. $$
Here 
$x (n) = x \ot t^n$ for $x\in \mf{g}$, 
$(, )$ is the normalized Killing form on $\mf g$
and $\underline c$ is the center. 
Define a $\Z$-grading on $\af$ by 
$\text{deg} \,x (n) = n$ and
$\text{deg} \, \underline c =0$. 

Set $\af_{>0} =\mf g \ot t \C [ t ]$,  
$\af_{<0}= \mf g \ot t^{-1} \C [ t^{-1} ]$,  
$\af_0 = \mf g \oplus \C \underline c$ 
and $ \af_{\geq 0} = \af_{>0} \oplus \af_0$. 
Denote the enveloping algebras of $\af, \af_{\geq 0}, \af_{<0}$
by $U, B, N$. 
Obviously $U, B, N$ inherit $\Z$-gradings from the corresponding Lie algebras. 

Define a character 
$$\gamma: \af_0 = \mf{g} \oplus \C \underline c \to \C; \quad 
\gamma|_{\mf g} = 0, \quad 
\gamma ( \underline c ) = 2h^\vee,$$
where $h^\vee$ is the dual Coxeter number of $\mf g$. 
It is easy to check that $\gamma$ 
is a semi-infinite character for $\af$ (see [S, Definition 1.1]). 

For any two $\Z$-graded vector spaces $M, M'$,
define the $\Z$-graded vector space $\h_\C (M, M')$ 
with homogeneous components 
$$\h_\C (M, M')_j = \{ f \in \text{Hom}_\C (M, M') | f (M_i) \subset M'_{i+j} \}. $$

The graded dual $\Nd = \oplus_i N_i^*$ of $N$
is an $N$-bimodule via the prescriptions 
$(n f ) (n_1) = f (n_1 n)$ and $ (f n) (n_1) = f( n n_1)$ 
for any $n, n_1\in N$, $f \in \Nd$. 
We have $\Nd = \h_\C (N, \C)$, 
if we equip $\C$ with the $\Z$-grading $\C = \C_0$. 

Consider the following sequence of isomorphisms of ($\Z$-graded) vector spaces: 
$$\h _B (U, \C_\gamma \ot_ \C B) 
\,\, \tilde \to \,\, 
\h _\C (N, B) 
\,\, \tilde \gets \,\,
\Nd \ot_\C B  
\,\, \tilde \to \,\, 
\Nd \ot _N U, $$
here $\C_\gamma$ is the one-dimensional representation of $\af_{\geq 0}$
defined by the character $\gamma: \af_0 \to \C$ and the surjection 
$\af_{\geq 0} \twoheadrightarrow \af_0$, and 
$\C_\gamma \ot_\C B$ is the tensor product of these two representations
as a left $\af_{\geq 0}$-module. 
In the leftmost term, 
$U$ is considered a $B$-module via left multiplication of $B$ on $U$, 
and $\h_B (U, \C_\gamma \ot_ \C B) $ is made into a (left) $U$-module
via the right multiplication of $U$ onto itself. 
The first isomorphism is defined as the restriction to $N$ using the identification 
$\C_\gamma \ot _\C B \tilde \to B; 1\ot b \mapsto b$. 

As a vector space, 
the semi-regular module $$S_\gamma = \Nd \ot_\C B. $$
It is also a $U$-bimodule: 
the left (resp. right) $U$-action on $S_\gamma$ 
is defined via the first two (resp. last) isomorphisms. 
The semi-infinite character $\gamma$ ensures that these two actions commute. 

\begin{lemma} \label{levelshift}
$\underline c \cd s = s \cd \underline c + 2h^{\vee} s $ 
for any $s\in S_\gamma$, 
where $\underline c \cd s$ and $s \cd \underline c$ 
stand for the left and right actions of $\underline c$ on $s\in S_\gamma$. 
\end{lemma}

\begin{proof}
Easily verified. 
\end{proof}

\begin{prop} {[S, Theorem 1.3]} \label {lrsr}
The map $\iota: \Nd \hookrightarrow S_\gamma; f \mapsto f \otimes 1$
is an inclusion of $N$-bimodules. 
The maps $ U \ot_N \Nd \to S_\gamma; u \ot f \mapsto u \cd \iota (f)$
and $\Nd \ot_N U \to S_\gamma; f \ot u \mapsto \iota (f) \cd u$
are bijections. 
\end{prop}

\begin{remark} %symmetry of the construction 
The sequence of isomorphisms 
$$S_\gamma = U \ot_N \Nd \cong B \ot_\C \Nd \cong \h_\C (N, B) 
\, \tilde \to \, \h_{B-\text {right}} (U, \C_{- \gamma} \ot B )$$
induces a right $U$-map from $S_\gamma$
to $\h_{B-\text {right}} (U, \C_{- \gamma} \ot B )$. 
The right $U$-module structure of the latter 
is given by the left multiplication of $U$ on the first argument in $\h$. 
\end{remark}

Let $P^+$ be the dominant integral weights of $\mf g$ and $\lambda \in P^+$. 
Denote by $V_{\lambda, k} = \text{Ind}_{ \af_{\geq 0}}^{\af} V_\lambda$
the Weyl module induced from the finite dimensional irreducible representation 
of $\mf g$ with highest weight $\lambda$ in level $k$. 
%should use circle-enclosed-star to denote the graded dual?
Let $V_{\lambda, k}^*$ be the graded dual of $V_{\lambda, k}$, 
on which $\af$ acts by $ X f (v) = - f ( X v)$ 
for any $X \in \af$, $f \in V_{\lambda, k}^*$ and $v \in V_{\lambda, k}$. 

Let $\mathcal M$ (resp. $\mathcal K$) 
denote the category of all $\Z$-graded representations of $\af$, 
which are over $N$ isomorphic to finite direct sums 
of may-be grading shifted copies of $N$ (resp. $\Nd$). 
In fact $\mathcal M$ (resp. $\mathcal K$) 
consists precisely of those $\Z$-graded $\af$-modules, 
which admit a finite filtration with factors isomorphic to Weyl modules 
(resp. the dual of Weyl modules) (see [S, Remarks 2.4]).

\begin{prop}{[S, Theorem 2.1]}
The functor $S_\gamma \ot_U -: \mathcal M \to \mathcal K$ 
defines an equivalence of categories with inverse $\h_U ( S_\gamma, - )$, 
such that short exact sequences correspond to short exact sequences. 
\end{prop}

\begin{proof}
Note that $S_\gamma \ot_U - \cong \Nd \ot_N -$ 
and $\h_U (S_\gamma, -) \cong \h_N (\Nd, - )$
by Proposition \ref{lrsr}. 
\end{proof}

\begin{prop} \label{tens}
Let $E$ be a $\Z$-graded $B$-module bounded from below, 
the functor $S_\gamma \ot_U -$ maps $U \ot_B E$ 
to $\h_B (U, \C_\gamma \ot E)$. 
\end{prop}

\begin{proof}
Similar to the construction of the semi-regular module $S_\gamma$,
consider the following sequence of isomorphisms of $\Z$-graded vector spaces: 
$$S_\gamma \ot_U ( U \ot_B E ) 
\,\, \cong \,\,           \Nd \ot_\C E 
\,\, \tilde \to \,\,       \h_\C ( N, E )
\,\, \tilde \gets \,\,   \h_B (U, \C_\gamma \ot E). $$
It is straightforward to check that, under these isomorphisms,
the (left) $U$-module structure of $S_\gamma \ot_U ( U \ot_B E )$ 
agrees with that of $ \h_B (U, \C_\gamma \ot E) $. 
\end{proof}

\begin{remark} \label{gtens}
In general for any $\Z$-graded $B$-module $E'$, 
the inclusion 
$S_\gamma \ot_U ( U \ot_B E' ) \cong \Nd \ot_\C E' 
\hookrightarrow \h_B (U, \C_\gamma \ot E' )$
is a $U$-map. 
%take E = B, yields the construction of S_\gamma as a left $U$-module
\end{remark}

\begin{prop} \label{hom}
Let $F$ be a $\Z$-graded $B$-module bounded from above, 
then the functor $\h_U (S_\gamma, - )$
maps $\h_B (U, F) $ 
to $U \ot_B (\C_{-\gamma} \ot F)$. 
\end{prop}

\begin{proof}
The isomorphism of vector spaces 
$U \ot_B (\C_{-\gamma} \ot F) \, \tilde \to \, \h_U (S_\gamma, \h_B (U, F) )$, 
induced from 
$$\h_U (S_\gamma, \h_B (U, F) )  
\cong \h_N (\Nd, \h_\C (N, F )) $$ $$
\cong \h_\C (\Nd, F)
\cong N \ot_\C F
\cong U \ot_B (\C_{-\gamma} \ot F), $$
agrees with the composition of (left) $U$-maps
$$U \ot_B (\C_{-\gamma} \ot F) \to 
\h_U (S_\gamma, S_\gamma \ot_U (U \ot_B (\C_{-\gamma} \ot F) ) ) %$$ $$
\to \h_U (S_\gamma, \h_B (U, F) ), $$
hence it is a $U$-isomorphism. 
\end{proof}

In particular $S_\gamma \ot_U -$ transforms Weyl modules to the dual of Weyl modules, 
and $\h_U (S_\gamma, -)$ transforms the latter to the former (both with a level shift). 

\begin{corollary} \label{tiny}
$S_\gamma \ot_U V_{\lambda, k} \cong V_{\lambda^*, \bar k}^*$, 
and $\h_U (S_\gamma, V_{\lambda, \bar k}^*) \cong V_{\lambda^*,  k}$, 
here $\lambda^*$ denotes the highest weight of $V_\lambda^*$. 
\end{corollary}

\begin{proof}
Note that $U \ot_B V_\lambda = V_{\lambda, k}$ 
and $\h_B (U, \C_\gamma \ot V_\lambda) \cong V_{\lambda^*, \bar k}^*$
if $\underline c$ acts on $V_\lambda$ as scalar multiplication by $k$. 
\end{proof}

%Section 3 starts here.................................................................................................
\section{Realization of $S_\gamma \ot_U \V$ inside $U^*$}

Fix a complex number $k$, 
and let $\V = U\A_{\mf g, k}$ 
be the vertex operator algebra associated to the vertex algebroid 
$\A_{\mf{g}, k}$ (see [AG], [GMS1, 2], [Z1]).  
Note that in [Z1], 
we used $\V$ to denote the vertex operator algebra for generic values of $k \notin \Q$, 
but here we adopt this notation with no restriction on $k$. 

The vertex operator algebra $\V$ admits two commuting actions of $\af$
in dual levels $k, \bar k = -2h^\vee -k$. 
It follows from Lemma \ref{levelshift} that $S_\gamma \ot _U \V$, 
using the $\af_k$-module structure of $\V$, 
becomes a $\af_{- \bar k} \oplus \af_{\bar k}$-representation. 
%(subscripts being the levels of the representations). 
Define $U(\af, k) = U(\af) / (\underline c - k) U(\af)$. %change notation to U_k??
Our goal is to construct  an embedding of $U$-bimodules 
$$\Phi: S_\gamma \ot _U \V \hookrightarrow U(\af, \bar k)^*. $$

Let $\B = \oplus_{i \leq 0} \B_i$ (denoted by ``$B$'' with opposite grading in [Z1]) 
be the commutative vertex subalgebra of $\V$ generated by $A$, 
where $A$ is the commutative algebra of regular functions on an affine connected algebraic group $G$ with Lie algebra $\mf g$. 
Recall that $\B$ is closed under the actions of $U(\af_{\geq 0}, k)$ and $U(\af_{\geq 0}, \bar k)$. 
As a $\af_k$-module, we have 
$\V \cong U \ot _B \B \cong N \ot_\C \B$
(see e.g. [Z1, Proposition 3.16]). 
Since $S_\gamma \cong \Nd \ot_N U$ as a right $U$-module, 
we have  
$$S_\gamma \ot_U \V \cong \Nd \ot_N \V \cong \Nd \ot_\C \B. $$

Define a functional $\ep: \B \to \C$ as follows: 
$\ep|_{\B _{\leq 1}} = 0$ and its restriction to $\B_0 = A$ 
is the evaluation of functions at identity. 

Multiplication induces isomorphism of vector spaces: 
$N \ot_\C B \cong U$, 
hence any $u \in U$ can be written as $u = u_{<0} u_{\geq 0}$ 
with $u_{<0} \in N$ and $u_{\geq 0} \in B$. 

Let $\bar{}: U \to U; u \to \overline u$
be the anti-involution of $U$
determined by $ - \text{Id}: \af \to \af$. 
Define a map 
$$\Phi: S \ot_U \V \to U^*$$
as follows: for any $f \in \Nd$, $b\in \B$, 
$$\Phi (f \ot b)  ( u_{<0} u_{\geq 0}) = f (\overline{ u_{<0}}) \ep ( u_{\geq 0}^r \cdot b), $$
here $u_{\geq 0}^r \cdot b$ means the $U(\af_{\geq 0}, \bar k)$-action on $\B$. 
In fact $\Phi (f \ot b) \in U(\af, \bar k) ^*$. 

The dual space $U^*$ is a $U$-bimodule via the recipes  
$ (u \cd g) (u_1) = g(u_1 u)$ and $ (g \cd u) (u_1) = g ( u u_1)$ 
for any $u, u_1\in U$, $g \in U^*$. 

\begin{theorem} \label{main}
For any $u\in U$ and $f \ot b \in \Nd \ot_\C \B$ ( $\cong S_\gamma \ot_U \V$), 
we have 
$$\Phi ( u^l \cd (f \ot b)) = ( \Phi (f \ot b)) \cd \bar u,  $$ $$
\Phi ( u^r \cd (f \ot b)) = u \cd ( \Phi (f \ot b)), $$
here $u^l \cd (f \ot b)$, $u^r \cd (f \ot b)$ 
stand for the $\af_{- \bar k}$- and $\af_{ \bar k}$-actions 
on $S_\gamma \ot_U \V$ respectively. 
\end{theorem}

To prove the theorem, we need some preparations. 
First, let 
$$\Th: S \ot_U \V \to \h_B (U, \C_\gamma \ot \B)$$
be the (left) $U$-map described in Remark \ref{gtens}
(taking $E' = \B$). 
Note that we regard $\V, \B$ as non-positively graded, 
i.e. taking the opposite of the grading 
defined by the conformal weights of the vertex operator algebra $\V$. 
Here $\B$ is regarded as a left $B$-module via the $U(\af_{\geq 0}, k)$-action on $\B$, 
and $\Th$ is a $U(\af, - \bar k)$-map. 

Following [GMS2, Z1], 
let $\tau_i$ be an orthonormal basis of $\mf g$ 
with respect to the normalized Killing form $(, )$. 
Let $C_{ijk}$ be the structure constants 
determined by $[ \tau_i, \tau_j ] = C_{ijk} \tau_k$. 
We identify $\mf g$ with the tangent space to the identity of $G$. 
Let  $\tau_i^L$ (resp. $\tau_i^R$) 
be the left (resp. right) invariant vector fields 
valued $ \tau_i$ (resp. $- \tau_i$) at the identity, 
there exist regular functions $a^{ij} \in A$
such that $\tau_i^R = a ^{ij} \tau_j^L$ 
and $\ep (a ^{ij} ) = - \delta_{ij}$.

\begin{lemma}\label{epbl}
Let $\beta: B \to B$ be the automorphism 
which restricts to $\af_{\geq 0}$ as 
$X \mapsto \gamma(X) + X$, 
then for any $u_{\geq 0} \in B$ and $b\in \B$, we have 
$\ep ( \beta (u_{\geq 0}) ^l \cd b) = \ep ( \overline{u_{\geq 0}} ^r \cd b)$, 
here $ \beta (u_{\geq 0}) ^l \cd b$,  $\overline{u_{\geq 0}} ^r \cd b$
denote the $U(\af_{\geq 0}, k)$- and $U(\af_{\geq 0}, \bar k)$-actions on $\B$. 
\end{lemma}

\begin{proof}
By [Z1, Lemma 3.14 (10)], we have 
$\tau_j (n) ^l \cd b = \sum_i \sum_{p \geq 0} a ^{ij}_{(-1-p)} \tau_i (n + p) ^r \cdot b$ 
for any $n\geq 0$, $b\in \B$. 
Since $\ep |_{B_{\geq 1}} = 0$, we have 
$\ep ( \tau_j (n) ^l \cd b) = \sum_i \ep ( a ^{ij}_{(-1)} \tau_i ( n) ^r \cd b) 
= \sum_i ( - \delta_{ij}) \ep (\tau_i (n ) ^r \cd b) = - \ep (\tau_j (n) ^r \cd b)$.
Since the $U(\af_{\geq 0}, k)$- and $U(\af_{\geq 0}, \bar k)$-actions on $\B$ commute, 
for any $u_{\geq 0} = \tau_{j_1} (n_1) \cdots \tau_{j_q} (n_q)$, we have
$
\ep ( \beta (u_{\geq 0}) ^l \cd b)
= \ep ( u_{\geq 0} ^l \cd b)
= \ep ( - \tau_{j_1}(n_1) ^r \cd (\tau_{j_2}(n_2) \cdots ) ^l \cd b)
= \ep ( (\tau_{j_2}(n_2) \cdots ) ^l \cd (- \tau_{j_1}(n_1) ) ^r \cd b)
= \ep ( (\tau_{j_3}(n_3) \cdots ) ^l \cd (- \tau_{j_2}(n_2) ) ^r \cd (- \tau_{j_1}(n_1) ) ^r \cd b)
= \cdots 
= \ep ( (- \tau_{j_q}(n_q)) ^r \cd \cdots ( - \tau_{j_1}(n_1)) ^r \cd b)
= \ep ( \overline { u_{\geq 0}} ^r \cd b).
$
We also have 
$\ep ( \beta ( \underline{c} ) ^l \cd b) = \ep ( ( \underline{c} + 2h^{\vee}) ^l \cd b)
= \ep ( (k + 2h^{\vee}) b ) = \ep ( - \bar k b) = \ep ( \bar {\underline{c}} ^r \cd b)$, 
hence the lemma is proved. 
\end{proof}

\begin{prop} \label{PhiTh}
For any $ f \ot b \in \Nd \ot_\C \B $,   
we have 
$ \Phi ( f \ot b ) = \ep \, \Th ( f \ot b ) \,\, \bar{}$. 
\end{prop}

\begin{proof}
By the definition of $\Th$, 
for any $u = u_{<0} u_{\geq 0} \in U$, 
we have 
$\Th ( f \ot b ) ( \bar u ) 
= \Th ( f \ot b ) ( \overline { u_{\geq 0}} \, \overline { u_{<0}})
= f ( \overline {u_{<0}} ) \beta ( \overline { u_{\geq 0}} ) ^l \cd b$. 
Then it follows from Lemma \ref{epbl} that 
$\ep \, \Th ( f \ot b) ( \bar u) 
= f ( \overline { u_{<0}} ) \ep ( u_{\geq 0} ^r \cd b )
= \Phi ( f \ot b) ( u)$. 
\end{proof}

\begin{corollary}
For any $u\in U$ and $f \ot b \in \Nd \ot_\C \B$, 
we have $\Phi ( u^l \cd ( f \ot b ) ) = ( \Phi ( f \ot b ) ) \cd \bar u$. 
\end{corollary}

\begin{proof}
Since $\Th$ is a (left) $U$-map, 
by Proposition \ref{PhiTh}, we have 
$$
\Phi ( u^l \cd (f \ot b) ) = \ep  \Th ( u^l \cd ( f \ot b ) ) \,\, \bar{}
= \ep ( u \cd \Th ( f \ot b ) ) \,\, \bar{} $$ $$
= \ep  \Th ( f \ot b ) r_u  \,\, \bar{} \,\,
= \ep  \Th ( f \ot b ) \,\, \bar{} \,\, l_{\bar u} 
= \Phi ( f \ot b ) \, l_{\bar u} 
= ( \Phi ( f \ot b ) ) \cd \bar u, 
$$
where $r_u, l_{\bar u}: U \to U$ 
denote the right and left multiplications by $u$ and $\bar u$ respectively. 
Hence we proved one half of Theorem \ref{main}. 
\end{proof}

Next we prove the other half of Theorem \ref{main}, 
which is to show that 
$$\Phi ( u^r \cd (f \ot b)) = u \cd ( \Phi (f \ot b)). $$

If $u = u_{\geq 0} \in B$, 
then $ {u_{\geq 0}} ^r \cd ( f \ot b ) = f \ot {u_{\geq 0}}^r \cd b$. 
Hence $ \Phi ( f \ot ( {u_{\geq 0}}^r \cd b ) ) ( u'_{<0} u'_{\geq 0} ) 
= f ( \overline{ u'_{<0} } ) \ep ( {u'_{\geq 0}}^r \cd {u_{\geq 0}}^r \cd b ) 
= \Phi ( f \ot b ) ( u'_{<0} u' _{\geq 0} u_{\geq 0} ) 
= u_{\geq 0} \cd ( \Phi ( f \ot b )) ( u'_{<0} u'_{\geq 0} )$, 
which means that 
$\Phi ( {u_{\geq 0}}^r \cd (f \ot b)) = u_{\geq 0} \cd ( \Phi (f \ot b))$. 

To prove it holds for $u = u_{<0} \in N$ as well, 
it suffices to show that 
$\Phi ( \tau_i (-1) ^r \cd (f \ot b)) = \tau_i (-1)  \cd ( \Phi (f \ot b))$
since $\af_{<0}$ is generated by $\af_{-1}$. 

Recall that although 
$\B$ is only closed under the action of $U(\af_{\geq 0}, \bar k)$, 
it can be equipped with a $\af_{\bar k}$-module structure 
$\tilde \rho: U \to \text{End} ( \B )$ such that 
$\tilde \rho ( u_{\geq 0} )  b = { u_{\geq 0} }^r \cd b$ 
for any $u_{\geq 0} \in B$ and $b\in \B$ 
(see [Z1, Lemma 3.29, Remark 3.30]). 
In addition, we have 
$$\tau_i (-1) ^r \cd ( f \ot b ) = \sum_j f \cd \tau_j (-1) \ot ( a^{ij} b ) + f \ot \tilde \rho( \tau_i (-1) ) b $$
(see [Z1, Lemma 3.14 (9)]). 
Hence for any $u_{<0} \in N$, $u_0\in U (\af_0)$ and $ u_{>0} \in U( \af_{>0} )$, 
we have 
$$
\Phi  ( \tau_i (-1) ^r \cd (f \ot b) ) ( u_{<0} u_0 u_{>0} )
$$
$$
= \sum_j f ( \tau_j (-1) \overline { u_{<0} } ) \ep ( { u_0 }^r \cd {u_{>0}}^r \cd  ( a^{ij} b )  ) 
+ f ( \overline { u_{<0} } ) \ep ( { u_0 }^r \cd {u_{>0}}^r \cd  \tilde \rho( \tau_i (-1)) b )
$$
$$
= \sum_j f ( \tau_j (-1) \overline { u_{<0} } ) \ep ( { u_0 }^r \cd  a^{ij}_{(-1)} {u_{>0}}^r \cd b)
+ f ( \overline { u_{<0} } ) \ep ( { u_0 }^r \cd [ u_{>0}, \tau_i (-1) ] ^r \cd b ). 
$$ 
The last equality is because 
$[ {u_{>0}}^r, a^{ij}_{(-1)} ] |_\B = 0$ (see [Z1, Lemma 3.14 (4)]), 
and 
$ [ u_{>0}, \tau_i (-1) ] \in B$,
$\ep |_{\B_{\geq 1}} = 0$. 

On the other hand, we have
$$
\tau_i (-1)  \cd ( \Phi (f \ot b)) ( u_{<0} u_0 u_{>0} ) 
= \Phi (f \ot b) ( u_{<0} u_0 u_{>0} \tau_i (-1) )
$$
$$
= \Phi (f \ot b) ( u_{<0} u_0 [ u_{>0}, \tau_i (-1)] + u_{<0} [ u_0, \tau_i (-1) ] u_{>0} + u_{<0}  \tau_i (-1) u_{\geq 0} )
$$
$$
= f ( \overline { u_{<0}} ) \ep ( {u_0}^r \cd [ u_{>0}, \tau_i(-1)]^r \cd b) 
+ \sum_s f ( \overline{ u_{<0} \tau_s (-1) } ) \ep ( F^{i, s} (u_0) ^r \cd {u_{>0}}^r \cd b )
$$
$$
+ f ( \overline{ u_{<0} \tau_i (-1) } ) \ep ( {u_{\geq 0}}^r \cd b )
$$
where $F^{i, s} : U (\af_0) \to U ( \af_0 )$ are maps such that 
$[u_0, \tau_i (-1) ] = \sum_s \tau_s (-1) F^{i, s} ( u_0 )$ for any $u_0 \in U(\af_0)$. 

Since $\tau_k^R ( a^{ij} ) = C_{kip} a^{pj}$, we have 
$[ \tau_k (0)^r, a^{ij}_{(-1)} ] = C_{kip} a^{pj}_{(-1)}$
(see [Z1, Lemma 3.14 (4)]). 
Compare it with the commutator 
$[ \tau_k (0), \tau_i(-1) ] = C_{kip} \tau_p (-1)$, 
it follows that 
$[ {u_0} ^r,  a^{ij}_{(-1)} ] = \sum_s a^{sj}_{(-1)} F^{i, s} (u_0)^r$. 
Hence we have 
$$
\sum_j f ( \tau_j (-1) \overline { u_{<0} } ) \ep ( { u_0 }^r \cd  a^{ij}_{(-1)} {u_{>0}}^r \cd b)
$$
$$
= \sum_j f ( \tau_j (-1) \overline { u_{<0} } ) \ep ( \sum_s a^{sj}_{(-1)} F^{i, s} (u_0)^r \cd {u_{>0}}^r \cd b)
+ \sum_j f ( \tau_j (-1) \overline { u_{<0} } ) \ep ( a^{ij}_{(-1)} {u_{\geq 0}}^r \cd b)
$$
$$
= \sum_j f ( \tau_j (-1) \overline { u_{<0} } ) \ep ( - F^{i, j} (u_0) ^r \cd {u_{>0}}^r \cd b) 
+ f ( \tau_i (-1) \overline { u_{<0} } ) \ep ( - {u_{\geq 0}}^r \cd b )
$$
$$
= \sum_j f ( \overline{ u_{<0} \tau_j (-1) } ) \ep ( F^{i, j} (u_0) ^r \cd {u_{>0}}^r \cd b )
+ f ( \overline{ u_{<0} \tau_i (-1) } ) \ep ( {u_{\geq 0}}^r \cd b ),
$$
which proves that 
$\Phi  ( \tau_i (-1) ^r \cd (f \ot b) ) = \tau_i (-1)  \cd ( \Phi (f \ot b))$. 
The proof of Theorem \ref{main} is now complete. 

\begin{remark} \label{explreal}
Following the notations in [Z1], 
let $\{ \widetilde \omega_i\}$ 
be right invariant $1$-forms dual to $\{ \tau_i^R \}$, 
and let $\widetilde { \B_0 } $ 
be the linear span of elements of the form
$\partial ^{(j_1)} \widetilde \omega_{i_1} \cdots \partial ^{(j_n)} \widetilde \omega_{i_n} $, 
then $\B = A \ot \widetilde { \B_0 } $. 
There is a non-degenerate pairing 
between $U(\af_{>0})$ and $ \widetilde { \B_0 } $, 
defined by $( u_{>0}, \tilde b ) = \ep ( {u_{>0}}^r \cd \tilde b )$, 
via which $ \widetilde { \B_0 } $ can be identified with  $U(\af_{>0})^ \circledast$, 
the graded dual of $U(\af_{>0})$. 
The regular functions $A$ can be identified with the Hopf dual $U(\mf g)^*_{\text {Hopf}}$, 
which is a subalgebra of $U(\mf g)^*$ defined by 
$$U(\mf g)^*_{\text {Hopf}} 
= \{ \phi \in U(\mf g)^* | \text{ Ker} \phi \text{ contains a two-sided ideal }  J \subset U(\mf{g}) 
$$ $$
\qquad \qquad \text{ of finite codimension} \}. $$
It is not hard to see that 
$\ep (u_0^r \cd u_{>0}^r \cd a \tilde b ) = \ep (u_0^r \cd a) \ep (u_{>0}^r \cd \tilde b)$
for any $u_0 \in U(\mf g), u_{>0} \in U(\af_{>0}), a\in A$ and $ \tilde b \in  \widetilde { \B_0 }$. 
Hence 
$$S \ot_U \V \cong \Nd \ot \B \cong \Nd \ot A \ot \widetilde { \B_0 } 
\cong U(\af_{<0})^\circledast \ot U(\mf{g})^*_{\text {Hopf}} \ot U(\af_{>0})^\circledast 
\subset U(\af, \bar k) ^*, 
$$
and $\Phi$ is injective. 
\end{remark}

%Section 4 starts here%%%%%%%%%%%%%%%%%%%%%%%%%%%%%%%%%
\section{filtrations of the vertex operator algebra $\V$}

Fix $k \in \Q$, $k > - h^\vee$; set $\varkappa = k + h^\vee > 0$. 
Let $\mathcal O_{- \varkappa}$ be the full subcategory 
of the category of $\af_{ \bar k}$-modules
defined by Kazhdan and Lusztig in [KL1-4]. 
They constructed a tensor structure on $\mathcal O_{- \varkappa}$, 
and established an equivalence of tensor categories 
between $\mathcal O_{- \varkappa}$ and the category 
of finite-dimensional integrable representations 
of the quantum group with quantum parameter $q = e^{- i\pi/ \varkappa}$ 
(in the simply-laced case). 

Let $V_{\lambda,\bar k} = \text{Ind}_{\af_{\geq 0}}^{\af} V_\lambda$ be a Weyl module, 
denote the irreducible quotient of $V_{\lambda, \bar k}$ by $L_{\lambda, \bar k} $. 

\begin{define}{[KL1, Definition 2.15]}
$\mathcal O_{- \varkappa}$ is the full subcategory 
of $\af_{ \bar k}$-modules, which admits a finite composition series 
with factors of the form $L_{\lambda, \bar k}$ for various $\lambda \in P^+$. 
\end{define}

Let us recall some basic facts about $\mathcal O_{- \varkappa}$. 
The $\Z_{>0}$-grading on $\af_{>0}$ induces an $\N$-grading 
on the enveloping algebra: $U(\af_{>0}) = \bigoplus_{n \geq 0} U(\af_{>0})_n$. 
For any $V \in \mathcal O_{- \varkappa}$, $v \in V$, 
there exists an $n_1  \in \N$ such that $U(\af_{>0})_{n_1} \cd v = 0$. 

A module $\mathcal N$ over $\mf g \ot \C[t]$ is said to be a nil-module 
if $\text{dim}_\C \mathcal N < \infty$ and 
there exists a $n \geq 1$ such that $U(\af_{>0})_n  \mathcal N = 0$. 
Extend $\mathcal N$ to a $\af_{\geq 0}$-module by defining 
the action of $\underline c$ to be multiplication by $\bar k$, 
and let $\mathcal N_{\bar k} = \text{Ind}_{\af_{\geq 0}}^{\af} \mathcal N$
be the induced module. 
We say that $\mathcal N_{\bar k}$ is a generalized Weyl module. 

\begin{prop} {[KL1, Theorem 2.22]}
A $\af_{\bar k}$-module $V$ is in $\mathcal O_{ -\varkappa}$
if and only if $V$ is a quotient of a generalized Weyl module. 
\end{prop}

Given $V \in \mathcal O_{- \varkappa}$, 
let $\bar {L_0}: V \to V$ be the Sugawara operator defined by 
$\bar {L_0} v = - \frac{1}{\varkappa} \sum_{j > 0} \sum_i \tau_i(-j) \tau_i(j) v 
- \frac{1}{2 \varkappa}  \sum_i \tau_i(0) \tau_i(0) v $, 
where $\{ \tau_i \}$ is an orthonormal basis of $\mf g$ 
with respect to the normalized Killing form. 
Note that this operator is well defined and locally finite. 
Let $V_z$ be the generalized eigenspace of $\bar {L_0}$ with eigenvalue $-z \in \C$, 
we have $V = \bigoplus_{z \in \C} V_z$ 
with $\text{dim} V_z < \infty$. 
In fact there exist $z_1, \cdots, z_m \in \Q$ such that 
$\{ z | V_z \neq 0\} \subset \{ z_1 - \N \} \cup \cdots \cup \{ z_m - \N \}$, 
and $V$ becomes a $\Q$-graded $\af_{\bar k}$-representation, 
i.e. $x(n) V_z \subset V_{z+n}$ for any $x(n) \in \af$
(see [KL1, Lemma 2.20, Proposition 2.21]). 
In case $V = V_{\lambda, \bar k}$ is a Weyl module, 
$\bar {L_0}$ acts on $V_{\lambda, \bar k}$ semisimply. 
More specifically, we have 
$\bar {L_0} |_{U(\af_{<0})_{-n} \ot V_\lambda} 
= - \frac{ \langle \lambda, \lambda +2\rho \rangle }{2 \varkappa} + n$, 
where $\rho$ is the half sum of positive roots. 

Define the dual representation of $V$ as follows: 
as a vector space $V^* = \bigoplus_z (V_z)^*$; 
the $\af$-action is given by 
$ X f (v) = f ( -X v)$ for any $X \in \af, f\in V^*, v\in V$. 
In particular $V^*$ is a $\af_{- \bar k}$-module 
and locally $U(\af_{<0})$-finite. 
In order for $V^*$ to be a graded $\af$-module as well, 
set $(V^*)_z = (V_{-z} )^*$, 
or equivalently set $(V^*)_z$ to be the generalized $(-z)$-eigenspace of the operator 
$L_0' =  \frac{1}{\varkappa} \sum_{j > 0} \sum_i \tau_i(j) \tau_i(-j)  
+ \frac{1}{2 \varkappa}  \sum_i \tau_i(0) \tau_i(0)$
which acts on $V^*$. 

The contragredient dual $V^c$ is isomorphic to $V^*$ as a vector space, 
but instead of using $- \text{Id}: \af \to \af$, 
we use the anti-involution $x(n) \mapsto - x(-n), \underline c \mapsto \underline c$
to define the $\af$-action on $V^c$. 
Unlike $V^*$, 
the contragredient module $V^c$ is a $\af_{\bar k}$-representation, 
locally $U(\af_{>0})$-finite, 
and in fact belongs to $\mathcal O_{ -\varkappa}$. 

Given $V \in \mathcal O_{ -\varkappa}$, 
define a map $\phi_V: V^* \ot V \to U(\af, \bar k)^*; 
\phi_V ( f \ot v ) (u) = f ( u \cd v )$
for any $f \in V^*, v\in V, u \in U(\af)$. 
It is easy to see that $\phi_V$ is a $\af_{- \bar k} \oplus \af_{\bar k}$-map, 
where the $\af_{- \bar k} \oplus \af_{\bar k}$-module structure of $U(\af, \bar k)^*$
is given by $ (X, 0) \cd g = - g \cd X$
and $(0, X) \cd g = X \cd g$
for any $X \in \af$, $g \in U(\af, \bar k)^*$. 
Denote the image of $\phi_V$ by $\M (V)$, 
which is called the matrix coefficients of $V$. 

Recall the $\af_{ - \bar k} \oplus \af_{\bar k}$-map 
$\Phi: S_\gamma \ot_U \V \to U(\af, \bar k)^*$ defined in Section 3. 
As pointed out in Remark \ref{explreal}, 
the map $\Phi$ is injective and its image, 
which we denote by $\M^{\mathcal O_{- \varkappa}}$, 
is isomorphic to 
$U(\af_{<0})^\circledast \ot U(\mf g)^*_{\text{Hopf}} \ot U(\af_{>0})^\circledast$. 
Here $U(\af_{<0})^\circledast = \bigoplus_{n \leq 0} (U(\af_{<0})_n)^*$, 
$U(\af_{>0})^\circledast = \bigoplus_{n \geq 0} (U (\af_{<0})_n)^*$
are graded duals. 

\begin{prop}
$\M^{\mathcal O_{- \varkappa}}$ consists of matrix coefficients of modules from the category 
$\mathcal O_{- \varkappa}$, 
i.e. $\M^{\mathcal O_{- \varkappa}} = \sum_{V \in \mathcal O_{- \varkappa}} \M (V)$. 
\end{prop}

\begin{proof}
Let $V = \bigoplus_z V_z  \in \mathcal O_{- \varkappa}$, 
$v \in V$ and $f \in V^*$, 
for any $u = u_{<0} u_0 u_{>0} \in U = U(\af)$, 
we have $\phi_V ( f \ot v ) ( u) = \langle f, u_{<0} u_0 u_{>0} \cd v \rangle
=\langle \overline {u_{<0}} \cd f, u_0 \cd u_{>0} \cd v \rangle$. 
Since $V \in \mathcal O_{- \varkappa}$, 
there exist $n_1, n_2 \in \N$ such that 
$U(\af_{<0})_{ - n_1} \cd f = U( \af_{>0} )_{n_2} \cd v = 0 $. 
Moreover each $V_z$ is finite-dimensional and semisimple as a $\mf g$-module, 
therefore it is not hard to see that 
$\phi_V (f \ot v) \in U(\af_{<0})^\circledast \ot U(\mf g)^*_{\text{Hopf}} \ot U(\af_{>0})^\circledast$, 
i.e. $\M(V) \subset \M^{\mathcal O_{- \varkappa}}$. 

On the other hand, 
let $g \in  \M^{\mathcal O_{- \varkappa}}$, 
there exists an $n \in \N$ such that $U(\af_{>0})_{n} \cd g = 0$. 
Since each $U(\af_{>0})_{n'}$ is finite-dimensional
and $\mf g$ acts on $\M^{\mathcal O_{- \varkappa}}$ locally finitely, 
the $\af_{\geq 0}$-submodule generated by $g$ is a nil-module. 
Hence the $\af$-submodule $W = U(\af) \cd g$ 
generated by $g$ is a quotient of a generalized Weyl module, 
hence it belongs to $\mathcal O_{ - \varkappa}$. 
Let $\delta$ be the functional on $U^*$
defined by $\delta (g' ) = g' (1)$, 
then $\delta \in W^*$ 
and $g = \phi_W ( \delta \ot g ) \subset \M (W)$. 
\end{proof}

Define two operators $\bar {L_0}, L_0'$ 
that act on  $\M^{\mathcal O_{- \varkappa}}$ as follows: 
for any $g \in \M^{\mathcal O_{- \varkappa}}$, 
set $\bar {L_0} g =  - \frac{1}{\varkappa} \sum_{j > 0} \sum_i \tau_i(-j) \cd \tau_i(j) \cd g 
- \frac{1}{2 \varkappa}  \sum_i \tau_i(0) \cd \tau_i(0) \cd g$
and $L_0' g =  \frac{1}{\varkappa} \sum_{j > 0} \sum_i g \cd \tau_i (- j ) \cd \tau_i(  j )  
+ \frac{1}{2 \varkappa}  \sum_i g \cd \tau_i(0) \cd \tau_i(0)  $. 
Let $\M^{\mathcal O_{- \varkappa}}_{z', z}$ be the subspace 
consisting of all $g \in \M^{\mathcal O_{- \varkappa}}$ 
such that $g$ is in the kernel of some power of $\bar {L_0} + z \text{Id}$ 
and the kernel of some power of $L_0' + z' \text{Id}$. 
Then  $\M^{\mathcal O_{- \varkappa}} = \bigoplus_{z, z'}  \M^{\mathcal O_{- \varkappa}}_{z', z}$, 
and $\phi_V ( (V^*)_{z'} \ot V_z ) \subset \M^{\mathcal O_{- \varkappa}}_{z', z}$
for any $V \in \mathcal O_{-\varkappa}$. 
Moreover $ \M^{\mathcal O_{- \varkappa}}_{z', z} \cd x(n) 
\subset  \M^{\mathcal O_{- \varkappa}}_{z' +n, z}$ and 
$ x(n) \cd \M^{\mathcal O_{- \varkappa}}_{z', z} \subset
 \M^{\mathcal O_{- \varkappa}}_{z', z+ n}$
 for any $x (n) \in \af$. 
%z,z' consistent with actions of \af \oplus \af
Define a $\Z$-grading on 
$\M^{\mathcal O_{- \varkappa}}$: 
for any $g_1 \in (U(\af_{<0})_n)^*$, 
$a \in U(\mf g)^*_{\text{Hopf}}$, 
$g_2 \in (U(\af_{>0})_{n'})^*$, 
define $\text{deg } g_1 \ot a \ot g_2 = - n - n'$; 
set $\M^{\mathcal O_{- \varkappa}}_{\,\, n} = \{ g | \text{ deg } g = n \}$. 
It is not difficult to see that
$\M^{\mathcal O_{- \varkappa}}_{\,\, n} 
= \bigoplus_{z + z' = n} \M^{\mathcal O_{- \varkappa}}_{z', z}$. 

Following [KL1, 3.3], 
define  a partial order on $P^+$ as follows: 
$\lambda \leq \mu$ if either $\lambda = \mu$ 
or $\langle \lambda, \lambda + 2 \rho \rangle < \langle \mu, \mu + 2 \rho \rangle$. 
Let $\mathcal O_{- \varkappa}^s$ be the full subcategory of $\mathcal O_{- \varkappa}$
whose objects are the $V$ in $\mathcal O_{- \varkappa}$
such that the composition factors of $V$ are of the form $L_{\lambda, \bar k}$
for some $\lambda$ in the finite set 
$F^s = \{ \lambda \in P^+ | \langle \lambda, \lambda + 2 \rho \rangle \leq s \}$. 

We say that a module $V \in \mathcal O_{- \varkappa}$ is  tilting 
if both $V$ and $V^c$ have a Weyl filtration. 
For any $\lambda \in P^+$, 
there exists an indecomposable tilting module $T_{\lambda, \bar k}$
such that $V_{\lambda, \bar k} \hookrightarrow T_{\lambda, \bar k}$, 
and any other Weyl modules $V_{\mu, \bar k}$ 
entering the Weyl filtration of $T_{\lambda, \bar k}$
satisfy $\mu < \lambda$
(see [KL4, Proposition 27.2]).

\begin{lemma} \label{(contra)weylfil}
Let $V, V' \in \mathcal O_{ -\varkappa}$. 
\begin{enumerate}
\item 
If $V$ has a (finite) Weyl filtration with factors 
isomorphic to $V_{\lambda_i, \bar k}$ for various $\lambda_i \in P^+$, 
then $\M (V) \subset \sum_i \M (T_{\lambda_i, \bar k})$. 
\item 
If $V'$ has a (finite) filtration with factors isomorphic to 
$V_{\mu_i, \bar k}^c$ for various $\mu_i \in P^+$, 
then $\M(V') \subset \sum_i \M (T_{\mu_i, \bar k}^c)$. 
\end{enumerate}
\end{lemma}

\begin{proof}
The proof is exactly the same as that of [Z2, Lemma 3.2]: 
we can construct an injection $V \hookrightarrow \bigoplus_{i} T_{\lambda_i, \bar k}$, 
and a surjection $\bigoplus_i T_{\mu_i, \bar k}^c \twoheadrightarrow V'$, 
since $\text{Ext}_{\mathcal O_{- \varkappa}}^1 ( V_{\lambda, \bar k}, V_{\mu, \bar k}^c ) = 0$
(see [KL4, Proposition 27.1]). 
\end{proof}

\begin{corollary}
$\M^{\mathcal O_{- \varkappa}}$ consists of the matrix coefficients of tilting modules from  
$\mathcal O_{- \varkappa}$, 
i.e. $\M^{\mathcal O_{- \varkappa}} 
= \sum_{V \in \mathcal O_{- \varkappa}, V \text{tilting}} \M (V)$. 
\end{corollary}

\begin{proof}
For any $V \in \mathcal O_{- \varkappa}$, 
choose $s$ such that $V \in \mathcal O_{- \varkappa}^s$. 
By [KL1, Proposition 3.9], 
there exists a $P$, projective in $\mathcal O_{- \varkappa}^s$
and having a (finite) Weyl filtration, 
such that $V$ a quotient of $P$. 
Hence by Lemma \ref{(contra)weylfil} (1), 
we have $\M (V) \subset \M (P) \subset \sum_i \M (T_{\lambda_i, \bar k})$
for some $\lambda_i \in F^s$. 
\end{proof}

\begin{prop}
Order the dominant weights in such a way 
$P^+ = \{ \nu_1,  \cdots, \nu_i, \cdots \}$
that $\nu_i < \nu_j$ implies $i < j$. 
Set $\M^{\mathcal O_{- \varkappa}, i} = \sum_{j \leq i} \M (T_{\nu_j, \bar k} )$, 
then $\M^{\mathcal O_{- \varkappa}, 1} \subset \cdots 
\subset \M^{\mathcal O_{- \varkappa}, i-1} 
\subset \M^{\mathcal O_{- \varkappa}, i}
\subset \cdots$
is an increasing filtration of $\af_{- \bar k} \oplus \af_{\bar k}$-submodules 
of $\M^{\mathcal O_{- \varkappa}}$ 
with factors $\M^{\mathcal O_{- \varkappa}, i} / \M^{\mathcal O_{- \varkappa}, i-1}$  
isomorphic to $V_{\nu_i, \bar k}^* \ot V_{- \omega_0 \nu_i, \bar k}^c$, 
where $\omega_0$ is the longest element in the Weyl group. 
\end{prop}

\begin{proof}
The proof is the same as that of  [Z2, Theorem 3.3], using Lemma \ref{(contra)weylfil}. 
\end{proof}

\begin{remark} \label{decomposition}
The category $\mathcal O_{- \varkappa}$ is a direct sum of subcategories 
corresponding to the orbits of the shifted action of affine Weyl group on the weight lattice
(see [KL4, Lemma 27.7]). 
Hence we can decompose  $\M^{\mathcal O_{- \varkappa}}$, 
as a $\af_{- \bar k} \oplus \af_{\bar k}$-module, 
into summands corresponding to the orbits as well. 
Some summands are semisimple
(see [KL4, Proposition 27.4], [Z2, Proposition 3.1]), 
but all have an increasing filtration of the above type. 
\end{remark}

\begin{prop}
The vertex operator algebra $\V$ is isomorphic to 
$\h_U (S_\gamma, \M^{\mathcal O_{- \varkappa}} )$ 
as a $\af_k \oplus \af_{\bar k}$-module. 
\end{prop}

\begin{proof}
Recall that $\M^{\mathcal O_{- \varkappa}} \cong S_\gamma \ot_U \V = \Nd \ot \B$. 
Hence $\h_U (S_\gamma, \M^{\mathcal O_{- \varkappa}} ) 
\cong \h_N (\Nd,  \Nd \ot \B) 
\cong \h_\C ( \Nd, \B )
\cong N \ot \B
\cong \V$, 
the second to last isomorphism is because $\B$ is non-positively graded
while $\Nd$ is non-negatively graded. 
Moreover the induced isomorphism 
$\V \to \h_U (S_\gamma, \M^{\mathcal O_{- \varkappa}} ) 
\cong \h_U (S_\gamma, S_\gamma \ot_U \V)$
is a $\af \oplus \af$-map. 
\end{proof}

\begin{lemma} \label{small}
For any $b \in \B$, 
there exists an $i$ such that $\Nd \ot b \subset \M^{\mathcal O_{- \varkappa}, i}$. 
\end{lemma}

\begin{proof}
For any $ f \in \Nd$ and $u_{\geq 0} \in U(\af_{\geq 0})$, 
we have $u_{\geq 0} \cd (f \ot b) = f \ot ( u_{\geq 0}^r \cd b )$. 
Let $\mathcal N$ be the $U(\af_{\geq 0}, \bar k)$-submodule of $\B$ 
generated by $b$, 
then $\mathcal N$ is a nil-module 
and the $\af_{\bar k}$-submodule $U(\af, \bar k) \cd (f \ot b)$ 
generated by $f \ot b$ is a quotient of the generalized Weyl module $\mathcal N_{\bar k}$. 
Hence $f \ot b \in \M (\mathcal N_{\bar k})$ for any $f \in \Nd$, 
hence there exists an $i$ 
such that $\Nd \ot b \subset \M^{\mathcal O_{- \varkappa}, i}$. 
\end{proof}

\begin{theorem} \label{main2}
Set $\Sigma^i = \h_U (S_\gamma, \M^{\mathcal O_{- \varkappa}, i} )$, 
then $\V = \bigcup_i \Sigma^i$ 
and $\Sigma^1 \subset \cdots \subset \Sigma^{i-1} \subset \Sigma^i \subset \cdots $
is an increasing filtration of $\af_k \oplus \af_{\bar k}$-submodules of $\V$
with factors $\Sigma^i / \Sigma^{i-1}$ isomorphic to
$V_{ -\omega_0 \nu_i, k} \ot V_{ - \omega_0 \nu_i, \bar k}^c$. 
\end{theorem}

\begin{proof}
For any $u_{<0} \ot b \in N \ot \B \cong \V$, 
let $\mathcal N' \subset \B $ be the $U(\af_{\geq 0}, k)$-submodule generated by $b$, 
then $\mathcal N' $ is finite-dimensional. 
For any $s \in S_\gamma$
we have $p ( s \ot (u_{<0} \ot b ) ) \in \Nd \ot \mathcal N'$,
where $p: S_\gamma \ot \V \to S_\gamma \ot_U \V$ is the canonical projection. 
By Lemma \ref{small}, 
there exists an $i$ 
such that $p ( s \ot (u_{<0} \ot b ) ) \in \M^{\mathcal O_{- \varkappa}, i}$
for any $s \in S_\gamma$, 
hence $u_{<0} \ot b \in \h_U (S_\gamma, \M^{\mathcal O_{- \varkappa}, i}) = \Sigma^i$. 
This proves that $\V = \bigcup_i \Sigma^i$. 

Note that $\M^{\mathcal O_{- \varkappa}, i}
= \bigoplus_{z', z} \M^{\mathcal O_{- \varkappa}, i}_{z', z}$
with $\text{dim } \M^{\mathcal O_{- \varkappa}, i}_{z', z} < \infty$. 
Fix $z$, 
the exact sequence of $\af_{-\bar k} \oplus \af_{\bar k}$-modules
$0 \to \M^{\mathcal O_{- \varkappa}, i-1} 
\to \M^{\mathcal O_{- \varkappa}, i} 
\to V_{\nu_i, \bar k}^* \ot V_{ - \omega_0 \nu_i, \bar k}^c \to 0$ 
restricts to an exact sequence of $\af_{- \bar k}$-modules
$0 \to \bigoplus_{z'} \M^{\mathcal O_{- \varkappa}, i-1}_{z', z} 
\to \bigoplus_{z'} \M^{\mathcal O_{- \varkappa}, i}_{z', z}
\to V_{\nu_i, \bar k}^* \ot (V_{ - \omega_0 \nu_i, \bar k}^c)_z \to 0$. 
Since $V_{\nu_i, \bar k}^* \ot (V_{ - \omega_0 \nu_i, \bar k}^c)_z$
is isomorphic to a finite direct sum of grading-shifted copies of $\Nd$ over $N$, 
by induction on $i$, 
so does $\bigoplus_{z'} \M^{\mathcal O_{- \varkappa}, i}_{z', z}$ for each $i$, 
and the two exact sequences split over $N$, 
which means that there exists a grading preserving $N$-map 
$V_{\nu_i, \bar k}^* \ot V_{ - \omega_0 \nu_i, \bar k}^c $
$\to \M^{\mathcal O_{- \varkappa}, i} $
so that its composition with the projection is identity on the former. 
Therefore the sequence of $\af_k \oplus \af_{\bar k}$-modules
$0 \to \h_U (S_\gamma, \M^{\mathcal O_{- \varkappa}, i-1} )
\to \h_U (S_\gamma, \M^{\mathcal O_{- \varkappa}, i} )
\to \h_U (S_\gamma, V_{\nu_i, \bar k}^* \ot V_{ - \omega_0 \nu_i, \bar k}^c) \to 0$
is exact since $\h_U (S_\gamma, - ) \cong \h_N (\Nd, -)$. 
Hence we have $\Sigma^i / \Sigma^{i-1} \cong 
\h_U (S_\gamma, V_{\nu_i, \bar k}^* \ot V_{ - \omega_0 \nu_i, \bar k}^c)$, 
which is isomorphic to 
$V_{ - \omega_0 \nu_i, k} \ot V_{ - \omega_0 \nu_i, \bar k}^c$
by Proposition \ref{hom}, Lemma \ref{tiny} and the fact that 
the grading on $V_{ - \omega_0 \nu_i, \bar k}^c$ is bounded from above. 
\end{proof}

\begin{remark}
The decomposition of $\M^{\mathcal O_{- \varkappa}}$ 
discussed in Remark \ref{decomposition}
leads to a decomposition of $\V$, as a $\af_k \oplus \af_{\bar k}$-module, 
into summands %of  $\af_k \oplus \af_{\bar k}$-submodules
corresponding to the orbits of the affine Weyl group on the weight lattice. 
Again some summands are semisimple, 
but each has an increasing filtration of the above type. 
\end{remark}

\begin{corollary}
The vertex operator algebra $\V$ 
admits a decreasing filtration of $\af_k \oplus \af_{\bar k}$-submodules
$\V \supset \Xi_1 \supset \cdots \supset \Xi_{i-1} \supset \Xi_i \supset \cdots $
with factors $\Xi_{i-1} / \Xi_i $ isomorphic to 
$V_{- \omega_0 \nu_i, k}^c \ot V_{ - \omega_0 \nu_i, \bar k}$, 
and $\bigcap_i \Xi_i = 0$.
\end{corollary}

\begin{proof}
Let $L_0, \bar {L_0}: \V \to \V$
be the Sugawara operators associated 
to the $\af_k$- and $\af_{\bar k}$-actions on $\V$ respectively, 
i.e. $L_0 = \frac{1} { \varkappa} \sum_{j >0} \sum_i \tau_i (-j) \tau_i (j) 
+ \frac{1}{2 \varkappa} \sum_i \tau_i (0) \tau_i (0) $, 
and $\bar {L_0} = - \frac{1} {\varkappa} \sum_{j >0} \sum_i \bar {\tau_i} (-j) \bar {\tau_i} (j)
- \frac{1} {2 \varkappa} \sum_i \bar {\tau_i} (0) \bar {\tau_i} (0)$. 
Now we regard the vertex operator algebra $\V = \bigoplus_{n \geq 0} \V_n$ 
as non-negatively graded, 
then the sum $\mathcal L_0 = L_0 + \bar {L_0}$ is the gradation operator, 
i.e. $\mathcal L_0 |_{\V_n} = n \text {Id}$
(see [Z1, Proposition 3.20, 3.24]). 

Let $\V_{z_1, z_2}$ be the subspace consisting of $v \in \V$
such that $v$ is killed by some power of $L_0 - z_1 \text{Id}$
and some power of $\bar {L_0} - z_2 \text{Id}$. 
It follows from Theorem \ref{main2} that 
$\V = \bigoplus_{z_1, z_2} \V_{z_1, z_2}$
with $\text{dim }  \V_{z_1, z_2} < \infty$. 

Recall the symmetric non-degenerate bilinear form $\langle, \rangle: \V \times \V \to \C$ 
constructed in [Z1, Proposition 3.28]. 
It is shown to be compatible with the vertex operator algebra structure of $\V$, 
in particular we have $\langle x(n) \cd, \cd \rangle = \langle \cd , - x(-n)  \cd \rangle$
and $\langle \bar y(n) \cd, \cd \rangle = \langle \cd,  - \bar y(-n) \cd \rangle$
for any $x (n) \in \af_k, \bar y(n) \in \af_{\bar k}$. 
It implies that 
$\langle L_0 \cd , \cd  \rangle = \langle \cd , L_0 \cd  \rangle$
and $\langle \bar {L_0} \cd , \cd  \rangle = \langle \cd , \bar {L_0} \cd  \rangle$. 
Hence $\langle, \rangle|_{\V_{z_1, z_2} \times \V_{z_1', z_2'}} = 0$
except when $z_1 = z_1'$ and $z_2 = z_2'$, 
in which case the pairing is non-degenerate. 

Let $\V^c = \bigoplus_{z_1, z_2} \V_{z_1, z_2}^*$
be the contragredient dual of $\V$, 
where the $\af_k$- and  $\af_{\bar k}$-actions on $\V^c$ 
are both defined by the anti-involution
$x (n) \mapsto - x(-n); \underline c \mapsto \underline c$ of $\af$. 
Then we have $\V \cong \V^c$ because of the bilinear form $\langle, \rangle$. 

Set $\Xi_i= \{ v \in \V \, | \, \langle v, \Sigma_i \rangle = 0 \}$, 
then $\Xi_i $ is a $\af_k \oplus \af_{\bar k}$-submodule of $\V$. 
Moreover we have $\Xi_i \subset \Xi_{i-1}$, 
and $\bigcap_i \Xi_i = 0$ 
because $\bigcup_i \Sigma_i = \V$ and $\langle, \rangle$ is non-degenerate. 
In fact 
$\Xi_{i-1} / \Xi_i \cong (\Sigma_i / \Sigma_{i-1})^c 
\cong ( V_{ -\omega_0 \nu_i, k} \ot V_{ - \omega_0 \nu_i, \bar k}^c )^c
\cong  V_{ -\omega_0 \nu_i, k}^c \ot V_{ - \omega_0 \nu_i, \bar k}$. 
\end{proof}

%CONSTRUCTION
%CONSTRUCTION

%\newpage 

\end{document}